\documentclass[notitlepage,11pt,reqno]{amsart}
\usepackage[foot]{amsaddr}
\usepackage{amssymb,nicefrac,bm,upgreek,mathtools,verbatim}
\usepackage[final]{hyperref}
\usepackage[mathscr]{eucal}
\usepackage{dsfont}
\usepackage[normalem]{ulem}
\usepackage{amsopn}

\usepackage[margin=1in]{geometry}
\allowdisplaybreaks
\newcommand{\stkout}[1]{\ifmmode\text{\sout{\ensuremath{#1}}}\else\sout{#1}\fi}
\newtheorem{theorem}{Theorem}[section]

\theoremstyle{definition}
\newtheorem{definition}{Definition}[section]

\theoremstyle{remark}

\numberwithin{theorem}{section}
\numberwithin{equation}{section}

\hypersetup{
  colorlinks=true,
  citecolor=mblue,
  linkcolor=mblue,
  urlcolor = blue,
  anchorcolor = blue,
  frenchlinks=false,
  pdfborder={0 0 0},
  naturalnames=false,
  hypertexnames=false,
  breaklinks}
\usepackage[capitalize,nameinlink]{cleveref}
\usepackage[abbrev,msc-links,nobysame,citation-order]{amsrefs} 
\crefname{section}{Section}{Sections}
\crefname{subsection}{Section}{Sections}
\crefname{condition}{Condition}{Conditions}
\crefname{hypothesis}{Hypothesis}{Conditions}
\crefname{assumption}{Assumption}{Assumptions}
\crefname{lemma}{Lemma}{Lemmas}
\crefname{fact}{Fact}{Facts}

\Crefname{figure}{Figure}{Figures}

\crefformat{equation}{\textup{#2(#1)#3}}
\crefrangeformat{equation}{\textup{#3(#1)#4--#5(#2)#6}}
\crefmultiformat{equation}{\textup{#2(#1)#3}}{ and \textup{#2(#1)#3}}
{, \textup{#2(#1)#3}}{, and \textup{#2(#1)#3}}
\crefrangemultiformat{equation}{\textup{#3(#1)#4--#5(#2)#6}}%
{ and \textup{#3(#1)#4--#5(#2)#6}}{, \textup{#3(#1)#4--#5(#2)#6}}%
{, and \textup{#3(#1)#4--#5(#2)#6}}

\Crefformat{equation}{#2Equation~\textup{(#1)}#3}
\Crefrangeformat{equation}{Equations~\textup{#3(#1)#4--#5(#2)#6}}
\Crefmultiformat{equation}{Equations~\textup{#2(#1)#3}}{ and \textup{#2(#1)#3}}
{, \textup{#2(#1)#3}}{, and \textup{#2(#1)#3}}
\Crefrangemultiformat{equation}{Equations~\textup{#3(#1)#4--#5(#2)#6}}%
{ and \textup{#3(#1)#4--#5(#2)#6}}{, \textup{#3(#1)#4--#5(#2)#6}}%
{, and \textup{#3(#1)#4--#5(#2)#6}}

\crefdefaultlabelformat{#2\textup{#1}#3}
\newcommand{\vertiii}[1]{{\left\vert\kern-0.25ex\left\vert\kern-0.25ex\left\vert #1 
    \right\vert\kern-0.25ex\right\vert\kern-0.25ex\right\vert}}
\newcommand{\Uadm}{\mathfrak U}
\newcommand{\Act}{\mathbb{U}}
\newcommand{\Usm}{\mathfrak U_{\mathsf{sm}}}
\newcommand{\Udsm}{\mathfrak U_{\mathsf{dsm}}} %deterministic stationary

\newcommand{\Um}{\mathfrak U_{\mathsf{m}}}
\newcommand{\Udm}{\mathfrak U_{\mathsf{dm}}}
\newcommand{\pV}{\mathrm{V}} % Probability measures space
\newcommand{\pv}{\mathrm{v}} % Probability measures

\newcommand{\sB}{{\mathscr{B}}}  % Ball
\newcommand{\cC}{{\mathcal{C}}}   % Continuous functions
\newcommand{\sE}{{\mathscr{E}}} 
\newcommand{\sF}{{\mathfrak{F}}}   % sigma field
\newcommand{\cJ}{{\mathcal{J}}}  % discounted cost
  
\newcommand{\sL}{{\mathscr{L}}}  % 
\newcommand{\Lp}{{L}}            % Lp
\newcommand{\Lyap}{{\mathcal{V}}}  % Lyapunov

\newcommand{\RR}{\mathds{R}}

\newcommand{\Rd}{{\mathds{R}^{d}}}
\DeclareMathOperator{\Exp}{\mathbb{E}}

\newcommand{\D}{\mathrm{d}}

\newcommand{\cD}{\mathcal{D}} %domain
 
\newcommand{\Sob}{{\mathscr W}}    % Sobolev Space
\newcommand{\Sobl}{{\mathscr W}_{\text{loc}}} % Sobolev Space(local)

\newcommand{\df}{:=}
\newcommand{\transp}{^{\mathsf{T}}}

\DeclareMathOperator*{\trace}{Tr}

\newcommand{\grad}{\nabla}
\newcommand{\uuptau}{{\Breve\uptau}}

\newcommand{\abs}[1]{\lvert#1\rvert}
\newcommand{\norm}[1]{\lVert#1\rVert}

\usepackage{color}
\definecolor{dmagenta}{rgb}{.4,.1,.5}
\definecolor{dblue}{rgb}{.0,.0,.5}
\definecolor{mblue}{rgb}{.0,.0,.7}
\definecolor{ddblue}{rgb}{.0,.0,.4}
\definecolor{dred}{rgb}{.7,.0,.0}
\definecolor{dgreen}{rgb}{.0,.5,.0}
\definecolor{Eeom}{rgb}{.0,.0,.5}

\begin{document}
\title[Near optimality of smooth policies]
{Near Optimality of Lipschitz and Smooth Policies in Controlled Diffusions}

\author[Somnath Pradhan]{Somnath Pradhan$^\dag$}
\address{$^\dag$Department of Mathematics and Statistics,
Queen's University, Kingston, ON, Canada}
\email{sp165@queensu.ca}

\author[Serdar Y\"{u}ksel]{Serdar Y\"{u}ksel$^{\ddag}$}
\address{$^\ddag$Department of Mathematics and Statistics,
Queen's University, Kingston, ON, Canada}
\email{yuksel@queensu.ca}

%%%%%%%%%%%%%%%%%%%%%%%%%%%%%%%%%%%%%%%%%%%%%%%%%%%%%%%%%%%%%%%%%%%%%%%%%%%%%%%%
\begin{abstract}
For optimal control of diffusions under several criteria, due to computational or analytical reasons, many studies have a apriori assumed control policies to be Lipschitz or smooth, often with no rigorous analysis on whether this restriction entails loss. While optimality of Markov/stationary Markov policies for expected finite horizon/infinite horizon (discounted/ergodic) cost and cost-up-to-exit time optimal control problems can be established under certain technical conditions, an optimal solution is typically only measurable in the state (and time, if the horizon is finite) with no apriori additional structural properties. In this paper, building on our recent work [S. Pradhan and S. Y\"uksel, Continuity of cost in Borkar control topology and implications on discrete space and time approximations for controlled diffusions under several criteria, Electronic Journal of Probability (2024)] establishing the regularity of optimal cost on the space of control policies under the Borkar control topology for a general class of controlled diffusions in $\Rd$, we establish near optimality of smooth or Lipschitz continuous policies for optimal control under expected finite horizon, infinite horizon discounted, infinite horizon average, and up-to-exit time cost criteria. Under mild assumptions, we first show that smooth/Lipschitz continuous policies are dense in the space of Markov/stationary Markov policies under the Borkar topology. Then utilizing the continuity of optimal costs as a function of policies on the space of Markov/stationary policies under the Borkar topology, we establish that optimal policies can be approximated by smooth/Lipschitz continuous policies with arbitrary precision. While our results are extensions of our recent work, the practical significance of an explicit statement and accessible presentation dedicated to Lipschitz and smooth policies, given their prominence in the literature, motivates our current paper. 
\end{abstract}
%%%%%%%%%%%%%%%%%%%%%%%%%%%%%%%%%%%%%%%%%%%%%%%%%%%%%%%%%%%%%%%%%%%%%%%%%%%%%%%%
\keywords{controlled diffusions, near optimality, smooth policy, Lipschitz continuous policy, Hamilton-Jacobi-Bellman equation}

\subjclass[2000]{Primary: 60J60, 93E20 Secondary: 35Q93}

%%%%%%%%%%%%%%%%%%%%%%%%%%%%%%%%%%%%%%%%%%%%%%%%%%%%%%%%%%%%%%%%%%%%%%%%%%%%%%%%
\maketitle
%\tableofcontents
%%%%%%%%%%%%%%%%%%%%%%%%%%%%%%%%%%%%%%%%%%%%%%%%%%%%%%%%%%%%%%%%%%%%%%%%%%%%%%%%

\section{Introduction and Importance of Lipschitz or Smooth Policies}

Smooth control policies have been imposed in many papers on optimal stochastic control under a variety of setups, as we detail further below, often with no rigorous analysis on whether the restriction to smooth policies entails loss in optimality or near-optimality.

With this motivation, our goal is to analyze the smooth approximations of optimal controls for controlled diffusions under various cost evaluation criteria. To achieve this, we leverage the regularity properties of the induced cost as a function of control policy under the Borkar topology  \cite{Bor89} \cite[Section~2.4]{ABG-book} (see also \cite{pradhan2022near} for further context). Then under this control topology, we establish that smooth or Lipschitz continuous policies form a dense subset within the space of Markov/stationary Markov policies. This denseness result, coupled with the continuity properties established in \cite{pradhan2022near}, allows us to derive general approximation results for optimal control policies. These results hold for a broad class of controlled diffusions evolving in the whole space $\Rd$\,.

%Due to its wide range of applications in domains that spans from  mathematical ﬁnance, large deviations and robust control, vehicle and mobile robot control and several other ﬁelds, the stochastic optimal control problems for controlled diffusions have been studied extensively in literature see, e.g., \cite{Bor-book}, \cite{HP09-book} (finite horizon cost) \cite{BS86}, \cite{BB96} (discounted cost) \cite{Ari-12}, \cite{AA13}, \cite{BG90}, \cite{BG88I}, \cite{BG90b}, \cite{ABG-book} (ergodic cost) and references therein\,. 

%However, the construction of optimal policies for controlled diffusion processes is a computationally challenging problem. 

Smooth or Lipschitz restrictions of policies play a crucial role in a variety of contexts:
\begin{itemize}
\item[•]\textbf{Existence of strong solutions under Lipschitz policies.} It is well known that under arbitrary measurable policies, existence of strong solutions to controlled stochastic differential equations is not guaranteed, and one typically needs to be content with weak solutions. However, Markov policies which are Lipschitz in their arguments, such as state, allow for the standard contraction based existence analysis for controlled stochastic differential equations \cite{fleming2012deterministic,CainesMeanField1,ABG-book,JZ-JMLR23}. 
%In controlled stochastic differential equations, e.g., (\ref{E1.1}), if the diffusion matrix $\upsigma(x)$ were also allowed to depend on control, in general existence in the relaxed control framework is a non-trivial problem, since Girsanov's change of measure arguments cannot be directly applied when the control is present in the diffusion term. However if we know that $\upsigma(\cdot, v(\cdot))$ is Lipschitz continuous for $v\in \Usm$ then $\cref{E1.1}$ admits unique strong solution. Since, in general stationary policies $v\in \Usm$ are just measurable functions. When $d\geq 3$, existence of suitable strong solutions in this setting is not known (see, \cite[Remarks~2.3.2]{ABG-book})\,. The existence of suitable strong solution (which is also a strong Markov process) under stationary Markov policies is essential to obtain stochastic representation of solutions of HJB equations (by applying It$\hat{o}$-Krylov formula) associated to stochastic optimal control problems.

\item[•]\textbf{Learning algorithms under Lipschitz policies.} The analysis of the near-optimality of smooth/Lipschitz continuous policies plays significant role in the development of learning algorithms for controlled diffusions in $\Rd$\,. As noted in \cite{lechner2022stability,richards2018lyapunov,chang2019neural}, Lipschitz regularity in policies facilitate efficient parametric representations and approximations in neural network based learning algorithms or episodic learning algorithms, and the Lipschitz restriction has been a common assumption on control policies. \cite{JZ-JMLR23} studies a $Q$-learning algorithm for controlled diffusion model under the assumption that admissible policies are Lipschitz continuous in the space variable and \cite{reisinger2023linear} studies policy gradient based learning under Lipschitz regularity conditions on policies. 

\item[•]\textbf{Robustness to noise approximations and model/data perturbations under Lipschitz policies.} For many practical systems the driving noise is modelled as an ideal Brownian, whereas in general this idealization is not accurate. Near optimality of Lipschitz continuous policies plays crucial role in establishing robustness result with respect to such incorrect noise modeling. In \cite{pradhanselkyuksel2023}, it is shown that within the class of Lipschitz continuous control policies, an optimal solution for a Brownian idealized model is near optimal for a true system driven by non-Brownian (but near-Brownian) noise, e.g., Wong-Zakai piecewise linear approximation, Karhunen-Lo\'eve approximations, mollified Brownian motion, fractional Brownian motion. Additionally, in a large variety of contexts, it is well documented that Lipschitz policies entail strong robustness properties in the presence of uncertainties in modeling and learning \cite{fazlyab2019efficient}.

\item[•]\textbf{Discrete-time approximations under Lipschitz policies.} Approximations are of significant practical consequence, they enable us to construct suitable numerical solutions to continuous-time control problems. Techniques involving Markov chain approximations aim to directly approximate a continuous-time state process using a discrete-time controlled Markov chain. The regularity properties of induced cost under Lipschitz continuous policies have significant implications for the construction of suitable discrete-time controlled Markov chain model to approximate controlled diffusion models, as one can show that invariant measures of discrete approximations converge to those of the continuous-limit under Lipschitz control policies and explicit convergence bounds can be obtained for the error between piece-wise constant path solutions and those of the continuous limit. By using near optimality of Lipschitz continuous policies, it can be shown that the discrete-time sampled controlled Markov chain and the corresponding value function asymptotically approximates the continuous time state process and the associated value function. Results along this direction appear in \cite{KD92} \cite{KH77}, \cite{KH01}; it can be shown that if one applies Lipschitz continuous policies such discrete-time Markov chain approximation results can be refined both with respect to the convergence of solution paths as well as the convergence of invariant measures (of discrete-time sampled chains to those of the continuous limit) under stationary Lipschitz policies.

\end{itemize}

%A commonly adopted approach of approximating controlled diffusions by a sequence of discrete time Markov chain via weak convergence methods was studied by Kushner et al., see \cite{KD92}, \cite{KH77}, \cite{KH01}\,. These studies focus on using numerical methods to create near optimal control policies for controlled diffusion models\,. This is achieved by approximating the space of relaxed control policies with piece-wise constant ones by applying standard chattering lemma (see \cite[Theorem~10.1.2, p. 278]{KD92}, \cite[Theorem~2.3.1]{ABG-book}) and by examining the weak convergence of probability measures on the path space to the measure on the continuous-time limit. 
%By using near optimality of Lipschitz continuous policies, it is shown in \cite{PY24}(\sy{our recent paper}) that the discrete-time sampled controlled Markov chain and the corresponding value function asymptotically approximates the continuous time state process and the associated value function.
%

%%%%%%%%%%%%%%%%%%%%%%%%%%%%%%%%%%%%%%%%%%%%%%
\subsection*{Contributions and main results} In this manuscript our goal is to study the following approximation problem: For a general class of controlled diffusions in $\Rd$ under what conditions one can approximate the optimal control policies for both finite/infinite horizon cost criteria by smooth/Lipschitz continuous policies?   

In order to address these questions, we first show that smooth/Lipschitz continuous polices are dense in the space of Markov/stationary Markov policies (see Theorem~\ref{TApproxSmoothPol1}). Then exploiting the regularity property of costs as a function of control policies (under Borkar topology \cite{Bor89}) obtained in \cite{pradhan2022near}, for both finite horizon and infinite horizon (discounted/ergodic), we establish these near optimality results (see Theorem~\ref{T1.1})\,. For the ergodic cost case we have studied this problem under two different additional assumptions, i.e., under either near-monotone type structural assumption on the running cost, or Lyapunov type stability assumption on the dynamics of the system. In particular, we establish the following:
\begin{itemize}
\item[•] We prove the denseness of smooth/Lipschitz continuous policies in the space of deterministic Markov/stationary Markov policies\,.
\item[•] We establish near-optimality of smooth/Lipschitz continuous policies for both finite and infinite horizon control problems for a general class of controlled diffusions in $\Rd$\,.
\end{itemize} 
%This approximation result provides a theoretical foundation for developing computationally efficient algorithms for finding approximate optimal policies for controlled diffusions.
%%%%%%%%%%%%%%%%%%%%%%%%%%%%%%%%%%%%%%%%%%%%%%%

The remaining part of the paper is organized as follows. In Section~\ref{PD} we provide the problem formulation\,. The denseness of smooth/Lipschitz continuous policies within the space of Markov/stationary Markov policies under Borkar topology is proved in Section~\ref{DensSmooth}. Finally, in Section~\ref{NearSmoothOpti}, using the denseness and continuity results we show that for finite and infinite horizon costs under the Borkar topology optimal policies can be approximated arbitrarily well by smooth/Lipschitz continuous policies\,.
%%%%%%%%%%%%%%%%%%%%%%%%%%%%%%%%%%%%%%%%%%%%%%%%%%%%%%%%%%%%%%%%%%%%%%%%%%%%%%%%
\subsection*{Notation:}
\begin{itemize}
\item For any set $A\subset\RR^{d}$, by $\uptau(A)$ we denote \emph{first exit time} of the process $\{X_{t}\}$ from the set $A\subset\RR^{d}$, defined by
\begin{equation*}
\uptau(A) \,\df\, \inf\,\{t>0\,\colon X_{t}\not\in A\}\,.
\end{equation*}
\item $\sB_{r}$ denotes the open ball of radius $r$ in $\RR^{d}$, centered at the origin,
\item $\uptau_{r}$, $\uuptau_{r}$ denote the first exist time from $\sB_{r}$, $\sB_{r}^c$ respectively, i.e., $\uptau_{r}\df \uptau(\sB_{r})$, and $\uuptau_{r}\df \uptau(\sB^{c}_{r})$.
\item By $\trace S$ we denote the trace of a square matrix $S$.
\item For any domain $\cD\subset\RR^{d}$, the space $\cC^{k}(\cD)$ ($\cC^{\infty}(\cD)$), $k\ge 0$, denotes the class of all real-valued functions on $\cD$ whose partial derivatives up to and including order $k$ (of any order) exist and are continuous.
\item $\cC_{\mathrm{c}}^k(\cD)$ denotes the subset of $\cC^{k}(\cD)$, $0\le k\le \infty$, consisting of functions that have compact support. This denotes the space of test functions.
\item $\cC_{b}(\Rd)$ denotes the class of bounded continuous functions on $\Rd$\,.
\item $\cC^{k}_{0}(\cD)$, denotes the subspace of $\cC^{k}(\cD)$, $0\le k < \infty$, consisting of functions that vanish in $\cD^c$.
\item $\cC^{k,r}(\cD)$, denotes the class of functions whose partial derivatives up to order $k$ are H\"older continuous of order $r$.
\item $\Lp^{p}(\cD)$, $p\in[1,\infty)$, denotes the Banach space
of (equivalence classes of) measurable functions $f$ satisfying
$\int_{\cD} \abs{f(x)}^{p}\,\D{x}<\infty$.
\item $\Sob^{k,p}(\cD)$, $k\ge0$, $p\ge1$ denotes the standard Sobolev space of functions on $\cD$ whose weak derivatives up to order $k$ are in $\Lp^{p}(\cD)$, equipped with its natural norm (see, \cite{Adams})\,.
\item  If $\mathcal{X}(Q)$ is a space of real-valued functions on $Q$, $\mathcal{X}_{\mathrm{loc}}(Q)$ consists of all functions $f$ such that $f\varphi\in\mathcal{X}(Q)$ for every $\varphi\in\cC_{\mathrm{c}}^{\infty}(Q)$. In a similar fashion, we define $\Sobl^{k, p}(\cD)$.
\item  For $\mu > 0$, let $e_{\mu}(x) = e^{-\mu\sqrt{1+\abs{x}^2}}$\,, $x\in\Rd$\,. Then $f\in \Lp^{p,\mu}((0, T)\times \Rd)$ if $fe_{\mu} \in \Lp^{p}((0, T)\times \Rd)$\,. Similarly, $\Sob^{1,2,p,\mu}((0, T)\times \Rd) = \{f\in \Lp^{p,\mu}((0, T)\times \Rd) \mid f, \frac{\partial f}{\partial t}, \frac{\partial f}{\partial x_i}, \frac{\partial^2 f}{\partial x_i \partial x_j}\in \Lp^{p,\mu}((0, T)\times \Rd) \}$ with natural norm (see \cite{BL84-book})
\begin{align*}
\norm{f}_{\Sob^{1,2,p,\mu}} =& \norm{\frac{\partial f}{\partial t}}_{\Lp^{p,\mu}((0, T)\times \Rd)} + \norm{f}_{\Lp^{p,\mu}((0, T)\times \Rd)} \\
& + \sum_{i}\norm{\frac{\partial f}{\partial x_i}}_{\Lp^{p,\mu}((0, T)\times \Rd)} + \sum_{i,j}\norm{\frac{\partial^2 f}{\partial x_i \partial x_j}}_{\Lp^{p,\mu}((0, T)\times \Rd)}\,. 
\end{align*} Also, we use the following convention $\norm{f}_{\Sob^{1,2,p,\mu}} = \norm{f}_{1,2,p,\mu}$\,.
\end{itemize}
%%%%%%%%%%%%%%%%%%%%%%%%%%%%%%%%%%%%%%%%%%%%%%%%%%%%%%%%%%%%%%%%%%%%%%%%%%%%%%%%
\section{Description of the problem}\label{PD} Let $\Act$ be a convex compact subset in a Euclidean space $\RR^m$, and $\pV=\mathscr{P}(\Act)$ be the space of probability measures on  $\Act$ with topology of weak convergence. Let $$b : \Rd \times \Act \to  \Rd, $$ $$ \sigma : \Rd \to \RR^{d \times d},\, \sigma = [\sigma_{ij}(\cdot)]_{1\leq i,j\leq d},$$ be given functions. We consider a stochastic optimal control problem whose state is evolving according to a controlled diffusion process given by the solution of the following stochastic differential equation (SDE)
\begin{equation}\label{E1.1}
\D X_t \,=\, b(X_t,U_t) \D t + \upsigma(X_t) \D W_t\,,
\quad X_0=x\in\Rd.
\end{equation}
Where 
\begin{itemize}
\item
$W$ is a $d$-dimensional standard Wiener process, defined on a complete probability space $(\Omega, \sF, \mathbb{P})$.
\item 
 We extend the drift term $b : \Rd \times \pV \to  \Rd$ as follows:
\begin{equation*}
b (x,\mathrm{v}) = \int_{\Act} b(x,\zeta)\mathrm{v}(\D \zeta), 
\end{equation*}
for $\mathrm{v}\in\pV$.
\item
$U$ is a $\pV$ valued process satisfying following non-anticipativity condition: for $s<t\,,$ $W_t - W_s$ is independent of
$$\sF_s := \,\,\mbox{the completion of}\,\,\, \sigma(X_0, U_r, W_r : r\leq s)\,\,\,\mbox{relative to} \,\, (\sF, \mathbb{P})\,.$$  
\end{itemize}
The process $U$ is called an \emph{admissible} control, and the set of all admissible controls is denoted by $\Uadm$ (see, \cite{BG90}). By a Markov control we mean an admissible control of the form $U_t = v(t,X_t)$ for some Borel measurable function $v:\RR_+\times\Rd\to\pV$. The space of all Markov controls is denoted by $\Um$\,. If the function $v$ is independent of $t$, i.e., $U_t = v(X_t)$ then $U$ or by an abuse of notation $v$ itself is called a stationary Markov control. The set of all stationary Markov controls is denoted by $\Usm$. A policy $v\in \Um$ is said to be deterministic Markov policy if $v(x) = \delta_{f(t,x)}$ for some measurable map $f:[0,\infty)\times\Rd\to \Act$. Let $\Udm$ be the space of all deterministic Markov policies\,. Similarly, a policy $v\in \Usm$ is said to be deterministic stationary Markov policy if $v(x) = \delta_{f(x)}$ for some measurable map $f:\Rd\to \Act$. Let $\Udsm$ be space of all deterministic stationary Markov policies\,.

To ensure existence and uniqueness of strong solutions of \cref{E1.1}, we impose the following assumptions on the drift $b$ and the diffusion matrix $\upsigma$\,. 
\begin{itemize}
%\item[(A1)]
\item[\hypertarget{A1}{{(A1)}}]
\emph{Local Lipschitz continuity:\/}
The function
$\upsigma\,=\,\bigl[\upsigma^{ij}\bigr]\colon\RR^{d}\to\RR^{d\times d}$,
$b\colon\Rd\times\Act\to\Rd$ are locally Lipschitz continuous in $x$ (uniformly with respect to the other variables for $b$). In other words, for some constant $C_{R}>0$
depending on $R>0$, we have
\begin{equation*}
\abs{b(x,\zeta) - b(y, \zeta)}^2 + \norm{\upsigma(x) - \upsigma(y)}^2 \,\le\, C_{R}\,\abs{x-y}^2
\end{equation*}
for all $x,y\in \sB_R$ and $\zeta\in\Act$, where $\norm{\upsigma}\df\sqrt{\trace(\upsigma\upsigma\transp)}$\,. Also, we are assuming that $b$ is jointly continuous in $(x,\zeta)$.

\medskip
\item[\hypertarget{A2}{{(A2)}}]
\emph{Affine growth condition:\/}
$b$ and $\upsigma$ satisfy a global growth condition of the form
\begin{equation*}
\sup_{\zeta\in\Act}\, \langle b(x, \zeta),x\rangle^{+} + \norm{\upsigma(x)}^{2} \,\le\,C_0 \bigl(1 + \abs{x}^{2}\bigr) \qquad \forall\, x\in\RR^{d},
\end{equation*}
for some constant $C_0>0$.

\medskip
\item[\hypertarget{A3}{{(A3)}}]
\emph{Nondegeneracy:\/}
For each $R>0$, it holds that
\begin{equation*}
\sum_{i,j=1}^{d} a^{ij}(x)z_{i}z_{j}
\,\ge\,C^{-1}_{R} \abs{z}^{2} \qquad\forall\, x\in \sB_{R}\,,
\end{equation*}
and for all $z=(z_{1},\dotsc,z_{d})\transp\in\RR^{d}$,
where $a\df \frac{1}{2}\upsigma \upsigma\transp$.
\end{itemize}

\subsection{The Borkar Topology on Control Policies}\label{B-topo}
We now introduce the Borkar topology on stationary or Markov controls \cite{Bor89}
\begin{itemize}
\item[•]\emph{Topology of Stationary Policies:}
From \cite[Section~2.4]{ABG-book}, we have that the set $\Usm$ is metrizable with compact metric.
\begin{definition}[Borkar topology of stationary Markov policies]\label{DefBorkarTopology1A}
A sequence $v_n\to v$ in $\Usm$ if and only if
\begin{equation}\label{BorkarTopology}
\lim_{n\to\infty}\int_{\Rd}f(x)\int_{\Act}g(x,\zeta)v_{n}(x)(\D \zeta)\D x = \int_{\Rd}f(x)\int_{\Act}g(x,\zeta)v(x)(\D \zeta)\D x
\end{equation}
for all $f\in L^1(\Rd)\cap L^2(\Rd)$ and $g\in \cC_b(\Rd\times \Act)$ (for more details, see \cite[Lemma~2.4.1]{ABG-book}, \cite{Bor89})\,.
\end{definition}
\item[•]\emph{Topology of Markov Policies:} In the proof of \cite[Theorem~3.1, Lemma~3.1]{Bor89}, replacing $A_n$ by $\hat{A}_n = A_n\times [0,n]$ and following the arguments as in proof of \cite[Theorem~3.1, Lemma~3.1]{Bor89}, we have the following topology on the space of Markov policies $\Um$\,. 
\begin{definition}[Borkar topology of Markov policies]\label{BKTP1}
A sequence $v_n\to v$ in $\Um$ if and only if
\begin{equation}\label{BorkarTopologyM}
\lim_{n\to\infty}\int_{0}^{\infty}\int_{\Rd}f(t,x)\int_{\Act}g(x,t,\zeta)v_{n}(t,x)(\D \zeta)\D x \D t = \int_{0}^{\infty}\int_{\Rd}f(t,x)\int_{\Act}g(x,t,\zeta)v(x,t)(\D \zeta)\D x \D t 
\end{equation}
for all $f\in L^1(\Rd\times [0, \infty))\cap L^2(\Rd\times [0, \infty))$ and $g\in \cC_b(\Rd\times [0, \infty)\times \Act)$\,.
\end{definition}
\end{itemize}
It is well known that under the hypotheses \hyperlink{A1}{{(A1)}}--\hyperlink{A3}{{(A3)}}, for any admissible control \cref{E1.1} has a unique strong solution \cite[Theorem~2.2.4]{ABG-book}, and under any stationary Markov strategy \cref{E1.1} has a unique strong solution which is a strong Feller (therefore strong Markov) process \cite[Theorem~2.2.12]{ABG-book}.

\subsection{Cost Criteria}
Let $c\colon\Rd\times\Act \to \RR_+$ be the \emph{running cost} function. We assume that $c$ is bounded, jointly continuous in $(x, \zeta)$ and locally Lipschitz continuous in its first argument uniformly with respect to $\zeta\in\Act$. We extend $c\colon\Rd\times\pV \to\RR_+$ as follows: for $\pv \in \pV$
\begin{equation*}
c(x,\pv) := \int_{\Act}c(x,\zeta)\pv(\D\zeta)\,.
\end{equation*}
In this article, we consider the problem of finding near-optimal smooth/Lipschitz continuous policies for finite horizon cost, cost up to an exit time, $\alpha$-discounted cost and ergodic cost, respectively:
\subsubsection{Finite Horizon Cost}
For $U\in \Uadm$, the associated \emph{finite horizon cost} is given by
\begin{equation}\label{FiniteCost1}
\cJ_{T}(x, U) = \Exp_x^{U}\left[\int_0^{T} c(X_s, U_s) \D{s} + c_{T}(X_T)\right]\,,
\end{equation} and the optimal value is defined as
\begin{equation}\label{FiniteCost1Opt}
\cJ_{T}^*(x) \,\df\, \inf_{U\in \Uadm}\cJ_{T}(x, U)\,.
\end{equation} where $T > 0$ is fixed, $c_{T}:\Rd\to \RR_{+}$ is the terminal cost and $X(\cdot)$ is the solution of \cref{E1.1} corresponding to $U\in\Uadm$ and $\Exp_x^{U}$ is the expectation with respect to the law of the process $X(\cdot)$ with initial condition $x$. The controller tries to minimize \cref{FiniteCost1} over his/her admissible policies $\Uadm$\,. Thus, a policy $U^*\in \Uadm$ is said to be optimal if we have 
\begin{equation}\label{FiniteCost1Opt1}
\cJ_{T}(x, U^*) = \cJ_{T}^*(x)\,.
\end{equation}
\subsubsection{Discounted Cost Criterion} 
For $U \in\Uadm$, the associated \emph{$\alpha$-discounted cost} is given by
\begin{equation}\label{EDiscost}
\cJ_{\alpha}^{U}(x, c) \,\df\, \Exp_x^{U} \left[\int_0^{\infty} e^{-\alpha s} c(X_s, U_s) \D s\right],\quad x\in\Rd\,,
\end{equation} where $\alpha > 0$ is the discounted factor\,. The controller tries to minimize \cref{EDiscost} over his/her admissible policies $\Uadm$\,. Thus, a policy $U^{*}\in \Uadm$ is said to be optimal if for all $x\in \Rd$ 
\begin{equation}\label{OPDcost}
\cJ_{\alpha}^{U^*}(x, c) = \inf_{U\in \Uadm}\cJ_{\alpha}^{U}(x, c) \,\,\, (\,=:\, \,\, V_{\alpha}(x))\,,
\end{equation} where $V_{\alpha}(x)$ is called the optimal value.
\subsubsection{Ergodic Cost Criterion}
For $U\in \Uadm$, the associated \emph{ergodic cost} is given by
\begin{equation}\label{ErgCost1}
\sE_{x}(c, U) = \limsup_{T\to \infty}\frac{1}{T}\Exp_x^{U}\left[\int_0^{T} c(X_s, U_s) \D{s}\right]\,.
\end{equation} and the optimal value is defined as
\begin{equation}\label{ErgCost1Opt}
\sE^*(c) \,\df\, \inf_{x\in\Rd}\inf_{U\in \Uadm}\sE_{x}(c, U)\,.
\end{equation}
Then a policy $U^*\in \Uadm$ is said to be optimal if we have 
\begin{equation}\label{ErgCost1Opt1}
\sE_{x}(c, U^*) = \sE^*(c)\,.
\end{equation}
%%%%%%%%%%%%%%%%%%%%%%%%%%%%%%%%%%%%%%%%%%%%%%%%%%%%%%%%%%%%%%%%%%%%%%%%%%%%%%%%

\subsubsection{Control up to an Exit Time}\label{exitTimeSection}
For each $U\in\Uadm$ the associated cost is given as
\begin{equation*}
\hat{\cJ}_{e}^{U}(x) \,\df \, \Exp_x^{U} \left[\int_0^{\tau(O)} e^{-\int_{0}^{t}\delta(X_s, U_s) \D s} c(X_t, U_t) \D t + e^{-\int_{0}^{\tau(O)}\delta(X_s, U_s) \D s}c_{e}(X_{\tau(O)})\right],\quad x\in\Rd\,,
\end{equation*}
where $O\subset \Rd$ is a smooth bounded domain, $\tau(O) \,\df\,  \inf\{t \geq 0: X_t\notin O\}$, $\delta(\cdot, \cdot): \bar{O}\times\Act\to [0, \infty)$ is the discount function and $c_{e}:\bar{O}\to \RR_+$ is the terminal cost function. The optimal value is defined as \[\hat{\cJ}_{e}^{*}(x)=\inf_{U\in \Uadm}\hat{\cJ}_{e}^{U}(x),\]
We assume that $\delta\in \cC(\bar{O}\times \Act)$, $c_{e}\in\Sob^{2,p}(O)$. 

%As in \cite{RZ21}, \cite[p.229]{B05Survey} the analysis leads to the following HJB equation. 
%\begin{align*}
%\min_{\zeta \in\Act}\left[\sL_{\zeta}\phi(x) - \delta(x, \zeta) \phi(x) + c(x, \zeta)\right] =  0\,,\quad \text{for all\ }\,\, x\in O\,,\quad\text{with}\quad
% \phi = h\,\,\, \text{on}\,\,\, \partial{O}\,.
%\end{align*} By similar argument as in \cite[Theorem~3.5.3]{ABG-book}, \cite[Theorem~3.5.6]{ABG-book} we have that $\hat{\cJ}_{e}^{*}$, $\hat{\cJ}_{e,n}^{*}$ are unique solutions to their respective HJB equations. Existence follows by utilizing the Leray-Schauder fixed point theorem as in \cite[Theorem~3.5.3]{ABG-book} and uniqueness follows by It$\hat{o}$-Krylov formula as in \cite[Theorem~3.5.6]{ABG-book}

\subsection{Problems Studied}

The main purpose of this manuscript will be to address the following problems:
\begin{itemize}
\item[•]\textbf{Near optimality of smooth policies.} For any given $\epsilon > 0$, whether it is possible to construct a Lipschitz continuous/ smooth policy $v_{\epsilon}$ such that
\begin{itemize}
\item[•]\emph{for finite horizon cost:} $\cJ_{T}(x, v_{\epsilon})\leq \cJ_{T}^*(x) + \epsilon$\,?
\item[•]\emph{for discounted cost:} $\cJ_{\alpha}^{v_{\epsilon}}(x, c)\leq V_{\alpha}(x) + \epsilon$ \,?
\item[•]\emph{for ergodic cost:} $\sE_{x}(c, v_{\epsilon})\leq \sE^*(c) + \epsilon$ \,?
\item[•] \emph{for cost up to an exit time:} $\hat{\cJ}_{e}^{v_{\epsilon}}(x) \leq \hat{\cJ}_{e}^{*}(x) + \epsilon$ \,?
\end{itemize}
\end{itemize}
In this manuscript, we have shown that under a mild set of assumptions the answers to the above mentioned questions are affirmative. 

Let us introduce a parametric family of elliptic operator, which will be useful in our analysis\,. With $\zeta\in \Act$ treated as a parameter, we define a family of operators $\sL_{\zeta}$ mapping
$\cC^2(\Rd)$ to $\cC(\Rd)$ by
\begin{equation}\label{E-cI}
\sL_{\zeta} f(x) \,\df\, \trace\bigl(a(x)\grad^2 f(x)\bigr) + \,b(x,\zeta)\cdot \grad f(x)\,, 
\end{equation}
where $f\in \cC^2(\Rd)\cap\cC_b(\Rd)$\, and for $\pv \in\pV$ we extend $\sL_{\zeta}$ as follows:
\begin{equation}\label{EExI}
\sL_\pv f(x) \,\df\, \int_{\Act} \sL_{\zeta} f(x)\pv(\D \zeta)\,.
\end{equation}Also, for each $v \in\Usm$, we define
\begin{equation}\label{Efixstra}
\sL_{v} f(x) \,\df\, \trace(a\grad^2 f(x)) + b(x,v(x))\cdot\grad f(x)\,.
\end{equation}

For our analysis of ergodic cost, we either assume that the running cost function $c$ is near-monotone with respect to $\sE^*(c)$, i.e.,
\begin{itemize}
\item[\hypertarget{A4}{{(A4)}}] It holds that
\begin{equation}\label{ENearmonot}
\liminf_{\norm{x}\to\infty}\inf_{\zeta\in \Act} c(x,\zeta) > \sE^*(c)\,,
\end{equation}
\end{itemize} or else, we assume the following Lyapunov stability condition on the dynamics.
\begin{itemize}
\item[\hypertarget{A5}{{(A5)}}]
There exists a positive constant $\widehat{C}_0$, and a pair of inf-compact  functions $(\Lyap, h)\in \cC^{2}(\Rd)\times\cC(\Rd\times\Act)$ (i.e., the sub-level sets $\{\Lyap\leq k\} \,,\{h\leq k\}$ are compact or empty sets in $\Rd$\,, $\Rd\times\Act$ respectively for each $k\in\RR$) such that
\begin{equation}\label{Lyap1}
\sL_{\zeta}\Lyap(x) \leq \widehat{C}_{0} - h(x,\zeta)\quad \text{for all}\,\,\, (x,\zeta)\in \Rd\times \Act\,,
\end{equation} where  $h$ ($>0$) is locally Lipschitz continuous in its first argument uniformly with respect to the second and $\Lyap > 1$.
\end{itemize} 

\section{Denseness of Smooth Policies}\label{DensSmooth}
In this section we show that under Borkar control topology, smooth/Lipschitz continuous policies are dense in $\Udsm$\,.
\begin{theorem}\label{TApproxSmoothPol1}
Let $v\in \Udsm$. Then there exist a sequence of Lipschitz continuous/ smooth policies $\{\Tilde{v}_{n, \eta}\}$ in $\Udsm$ such that as $\eta\to 0$ and $n\to \infty$, we have
\begin{equation}\label{ENearOptiSmooth1A}
\int_{\Rd} f(x)g(x, \Tilde{v}_{n, \eta}(x)) dx \to \int_{\Rd} f(x)g(x, v(x)) dx      
\end{equation} for all $f\in L^1(\Rd)\cap L^{2}(\Rd)$ and $g\in \cC_{b}(\Rd\times \Act)$\,.
\end{theorem}
 \begin{proof}
As in \cite[Theorem~4.2]{pradhan2022near}, by Lusin's theorem and Tietze's extension theorem we have there exists a sequence of continuous functions $\Tilde{v}_{n}$ such that as $n\to \infty$
\begin{equation}\label{ENearOptiSmooth1B}
\int_{\Rd} f(x)g(x, \Tilde{v}_{n}(x)) dx \to \int_{\Rd} f(x)g(x, v(x)) dx      
\end{equation} for all $f\in L^1(\Rd)\cap L^{2}(\Rd)$ and $g\in \cC_{b}(\Rd\times \Act)$\,.

Now, let $\phi \in \cC_c^{\infty}(\Rd)$ with $\phi = 0$ for $|x|>1$ and $\int_{\Rd}\phi(x) d x = 1$\,. Also, for each positive real number $\eta$, define $\phi_{\eta}(x) \df \eta^{-d}\phi(\frac{x}{\eta})$\,. Using this smooth function $\phi$ we define the following sequence of smooth polices $\Tilde{v}_{n, \eta}(x) := \phi_{\eta}* \Tilde{v}_{n}(x)$ (convolution with respect to $x$), for each $\eta>0$. Since, $\Tilde{v}_{n}$ is continuous, it is well known that as $\eta\to 0$ we have $\Tilde{v}_{n, \eta}(x) \to \Tilde{v}_{n}(x)$ for a.e $x\in \Rd$\, (in fact, this convergence is in the sense of uniform convergence over compact sets). Thus, for any $f\in L^1(\Rd)\cap L^{2}(\Rd)$ and $g\in \cC_{b}(\Rd\times \Act)$, by the dominated convergence theorem, as $\eta\to 0$ we deduce that
\begin{equation}\label{ENearOptiSmooth1C}
\int_{\Rd} f(x)g(x, \Tilde{v}_{n, \eta}(x)) dx \to \int_{\Rd} f(x)g(x, \Tilde{v}_{n}(x)) dx \,.     
\end{equation} 
Also, for each $\eta > 0$, it is easy to show that $\frac{\partial \Tilde{v}_{n, \eta}}{\partial x_i} = \frac{\partial \phi_{\eta}}{\partial x_i}* \Tilde{v}_{n}(x)$ is bounded for $i= 1,\dots , d$\,. This, in particular implies that for each $\eta >0$, $\Tilde{v}_{n, \eta}$ is uniformly Lipschitz continuous in  $\Rd$\,. 
By the triangle inequality, for any $f\in L^1(\Rd)\cap L^{2}(\Rd)$ and $g\in \cC_{b}(\Rd\times \Act)$ it is follows that
\begin{align}\label{ENearOptiSmooth1D}
& \abs{\int_{\Rd} f(x)g(x, \Tilde{v}_{n, \eta}(x)) dx - \int_{\Rd} f(x)g(x, v(x)) dx} \nonumber\\
& \leq  \abs{\int_{\Rd} f(x)g(x, \Tilde{v}_{n, \eta}(x)) dx - \int_{\Rd} f(x)g(x, \Tilde{v}_{n}(x)(x)) dx}\nonumber\\
& \qquad + \abs{\int_{\Rd} f(x)g(x, \Tilde{v}_{n}(x)(x)) dx - \int_{\Rd} f(x)g(x, v(x)) dx}  \,.   
\end{align}
Thus, in view of (\ref{ENearOptiSmooth1B}) and (\ref{ENearOptiSmooth1C}), we obtain our result\,.
\end{proof} 
Following the proof technique of Theorem~\ref{TApproxSmoothPol1}, considering time also as a separate parameter (taking mollifires which is function of both variables $t, x$) we have the following theorem.
\begin{theorem}\label{TApproxSmoothPol2}
Let $v\in \Udm$. Then there exist a sequence of Lipschitz continuous/ smooth policies $\{\Tilde{v}_{n, \eta}\}$ in $\Udm$ such that as $\eta\to 0$ and $n\to \infty$, we have
\begin{equation}\label{ENearOptiSmooth1AMarkov}
\int_{0}^{\infty}\int_{\Rd} f(t,x)g(t, x, \Tilde{v}_{n, \eta}(t,x)) dx dt \to \int_{0}^{\infty}\int_{\Rd} f(t,x)g(t, x, v(t,x)) dx dt      
\end{equation} for all $f\in L^1([0,\infty)\times \Rd)\cap L^{2}([0,\infty)\times\Rd)$ and $g\in \cC_{b}([0,\infty)\times\Rd\times \Act)$\,.
\end{theorem}
%%%%%%%%%%%%%%%%%%%%%%%%%%%%%%%%%%%%%%%%%%%%%
\section{Near Optimality of Smooth Policies}\label{NearSmoothOpti}
From \cite{pradhan2022near} we have the following continuity results of the various cost evaluation criteria as a function of the control policies\,.
\begin{theorem}\label{T1.1}
Suppose Assumptions \hyperlink{A1}{{(A1)}}--\hyperlink{A3}{{(A3)}} hold. Then, under Borkar control topology (see Definition~\ref{DefBorkarTopology1A}, \ref{BKTP1})
\begin{itemize}
\item[(i)]\emph{Finite horizon cost:} the map $v\mapsto\cJ_{T}(x, v)$\, from $\Um$ to $\RR$ is continuous\,.
\item[(ii)]\emph{Discounted cost:} the map $v\mapsto\cJ_{\alpha}^{v}(x, c)$\, from $\Usm$ to $\RR$ is continuous\,.
\item[(iii)]\emph{Ergodic cost:} (under additional assumption \hyperlink{A4}{{(A4)}} or \hyperlink{A5}{{(A5)}}) the map $v\mapsto\sE_{x}(c, v)$\, from $\Usm$ to $\RR$ is continuous\,.
\item[(iv)] \emph{Cost up to an exit time:} the map $v\mapsto\hat{\cJ}_{e}^{v}(x)$\, from $\Usm$ to $\RR$ is continuous\,.
\end{itemize}
\end{theorem}
\begin{proof}
(i) follows from \cite[Theorem~6.1]{pradhan2022near}, (ii) follows from \cite[Theorem~3.1]{pradhan2022near}, under \hyperlink{A4}{{(A4)}} (near-monotone type structural assumption on the running cost) (iii) follows from \cite[Theorem~3.5]{pradhan2022near}, under \hyperlink{A5}{{(A5)}} (Lyapunov stability assumption) (iii) follows from \cite[Theorem~3.8]{pradhan2022near}, and (iv) follows from \cite[Theorem~3.2]{pradhan2022near}\,.
\end{proof} 

The following theorem proves near-optimality of smooth/Lipschitz continuous policies for various cost evaluation criteria\,.
\begin{theorem}\label{T1.2}
Suppose Assumptions \hyperlink{A1}{{(A1)}}--\hyperlink{A3}{{(A3)}} hold. Then for any $\epsilon > 0$, we have 
\begin{itemize}
\item[(i)]\emph{Finite horizon cost:} there exists a smooth/Lipschitz continuous policy $v_{\epsilon}\in \Udm$ such that $$\cJ_{T}(x, v_{\epsilon})\leq \cJ_{T}^*(x) + \epsilon\,.$$
\item[(ii)]\emph{Discounted cost:} there exists a smooth/Lipschitz continuous policy $v_{\epsilon}\in \Udsm$ such that  $$\cJ_{\alpha}^{v_{\epsilon}}(x, c)\leq V_{\alpha}(x) + \epsilon\,.$$
\item[(iii)]\emph{Ergodic cost:} (under additional assumption \hyperlink{A4}{{(A4)}} or \hyperlink{A5}{{(A5)}}) there exists a smooth/Lipschitz continuous policy $v_{\epsilon}\in \Udsm$ such that $$\sE_{x}(c, v_{\epsilon})\leq \sE^*(c) + \epsilon\,.$$
\item[(iv)] \emph{Cost up to an exit time:} there exists a smooth/Lipschitz continuous policy $v_{\epsilon}\in \Udsm$ such that  $$\hat{\cJ}_{e}^{v_{\epsilon}}(x) \leq \hat{\cJ}_{e}^{*}(x) + \epsilon\,.$$
\end{itemize}
\end{theorem}
\begin{proof}
\emph{Finite horizon cost:} under the additional assumption that $b, a, \frac{\partial a}{\partial x_i}, c$ are uniformly bounded and the terminal cost $c_{T}\in \Sob^{2,p,\mu}(\Rd)\cap \Lp^{\infty}(\Rd)$\,,\,\, $p\ge 2$\, (see \cite[Section~6, p. 26]{pradhan2022near}), from \cite[Theorem~3.3, p. 235]{BL84-book}, we have that the optimality equation (or, the HJB equation)
\begin{align}\label{EfiniteHJB}
&\frac{\partial \psi}{\partial t} + \inf_{\zeta\in \Act}\left[\sL_{\zeta}\psi + c(x, \zeta) \right] = 0 \nonumber\\
& \psi(T,x) = c_{T}(x)
\end{align} admits a unique solution $\psi\in \Sob^{1,2,p,\mu}((0, T)\times\Rd)\cap \Lp^{\infty}((0, T)\times\Rd)$\,,\,\, $p\ge 2$\,. Thus, by It\^{o}-Krylov formula (see the standard verification results as in \cite[Theorem~3.5.2]{HP09-book}), we have there exists an optimal Markov policy (which is a measurable minimizing selector of the optimality equation (\ref{EfiniteHJB})), that is, there exists $v^*\in \Udm$ such that $\cJ_{T}(x, v^*) = \cJ_{T}^*(x)$\,. Also, from Theorem~\ref{T1.1}, we have that the map $v\mapsto\cJ_{T}(x, v)$\, from $\Um$ to $\RR$ is continuous\,. Thus, by the denseness of smooth/Lipschitz continuous policies (see Theorem~\ref{TApproxSmoothPol1}), for any $\epsilon > 0$ there exists a smooth policy $v_{\epsilon}$ such that
$$\cJ_{T}(x, v_{\epsilon})\leq \cJ_{T}(x, v^{*}) + \epsilon\,.$$ This proves $(i)$

\emph{Discounted cost:} from \cite[Theorem~3.5.6]{ABG-book}, we have that the optimal discounted cost $V_{\alpha}$ defined in \cref{OPDcost} is the unique solution in $\cC^2(\Rd)\cap\cC_b(\Rd)$ of the HJB equation
\begin{equation}\label{OptDHJB}
\min_{\zeta \in\Act}\left[\sL_{\zeta}V_{\alpha}(x) + c(x, \zeta)\right] = \alpha V_{\alpha}(x) \,,\quad \text{for all\ }\,\, x\in\Rd\,.
\end{equation}
Moreover, there exists $v^*\in \Udsm$ satisfying
\begin{equation}\label{OtpHJBSelc}
b(x,v^*(x))\cdot \grad V_{\alpha}(x) + c(x, v^*(x)) = \min_{\zeta\in \Act}\left[ b(x, \zeta)\cdot \grad V_{\alpha}(x) + c(x, \zeta)\right]\quad \text{a.e.}\,\,\, x\in\Rd
\end{equation} is $\alpha$-discounted optimal, that is, $\cJ_{\alpha}^{v^*}(x, c) = V_{\alpha}(x)$\,. In Theorem~\ref{T1.1}, we have that shown that the map $v\mapsto\cJ_{\alpha}^{v}(x, c)$\, from $\Usm$ to $\RR$ is continuous\,. Thus, by the density of smooth/Lipschitz continuous policies (see Theorem~\ref{TApproxSmoothPol1}), we deduce that for any $\epsilon > 0$ there exists a smooth/Lipschitz continuous policy $v_{\epsilon}$ such that
$$\cJ_{\alpha}^{v_{\epsilon}}(x, c)\leq V_{\alpha}(x) + \epsilon\,.$$ This proves $(ii)$\,.

\emph{Ergodic cost:} from \cite[Theorem~3.6.10]{ABG-book} (under \hyperlink{A4}{{(A4)}}), \cite[Theorem~3.7.11, Theorem~3.7.12]{ABG-book} (under \hyperlink{A5}{{(A5)}}) we know that there exists a unique solution pair $(V, \rho)\in \Sobl^{2,p}(\Rd)\times \RR$, \, $1< p < \infty$, with $V(0) = 0$ and $\inf_{\Rd} V > -\infty$ and $\rho \leq \sE^*(c)$, satisfying the ergodic HJB equation
\begin{equation}\label{EErgonearOpt1A}
\rho = \min_{\zeta \in\Act}\left[\sL_{\zeta}V(x) + c(x, \zeta)\right]\,.
\end{equation}
Moreover, there exist $v^*\in \Udsm$ satisfying
\begin{equation}\label{EErgonearOpt1B}
\min_{\zeta \in\Act}\left[\sL_{\zeta}V(x) + c(x, \zeta)\right] \,=\, \trace\bigl(a(x)\grad^2 V(x)\bigr) + b(x,v^*(x))\cdot \grad V(x) + c(x, v^*(x))\,,\quad \text{a.e.}\,\, x\in\Rd\,,
\end{equation} is an ergodic optimal policy, that is, $\sE_{x}(c, v^*)= \sE^*(c)$\,. Now, by the continuity of the map $v\mapsto\sE_{x}(c, v)$ from $\Usm$ to $\RR$ (see  Theorem~\ref{T1.1}), and the denseness of smooth/Lipschitz continuous policies (see Theorem~\ref{TApproxSmoothPol1}), we conclude that for any $\epsilon > 0$ there exists a smooth/Lipschitz continuous policy $v_{\epsilon}$ such that
$$\sE_{x}(c, v_{\epsilon})\leq \sE^*(c) + \epsilon\,.$$ This proves $(iii)$\,.

\emph{Cost up to an exit time:} from \cite[p.229]{B05Survey} (also, following the argument as in the proof of \cite[Theorem~3.5.3]{ABG-book}) we have that the HJB equation  
\begin{align}\label{EHJBExit}
\min_{\zeta \in\Act}\left[\sL_{\zeta}\phi(x) - \delta(x, \zeta) \phi(x) + c(x, \zeta)\right] =  0\,,\quad \text{for all\ }\,\, x\in O\,,\quad\text{with}\quad
 \phi = c_{e}\,\,\, \text{on}\,\,\, \partial{O}\,,
\end{align} admits unique solution $\phi\in \Sob^{2,p}(O)$, \, $1< p < \infty$ \,. Moreover, by It$\hat{o}$-Krylov formula (as in the proof of \cite[Theorem~3.5.6]{ABG-book}), we have there exists $v^*\in \Udsm$ (measurable minimizing selector of (\ref{EHJBExit})), such that $\hat{\cJ}_{e}^{v^*}(x) = \hat{\cJ}_{e}^{*}(x)$\,. Thus, by the continuity of the map $v\mapsto\hat{\cJ}_{e}^{v}(x)$ from $\Usm$ to $\RR$ (see  Theorem~\ref{T1.1}), and the denseness of smooth/Lipschitz continuous policies (see Theorem~\ref{TApproxSmoothPol1}), it follows that for any $\epsilon > 0$ there exists a smooth/Lipschitz continuous policy $v_{\epsilon}$ such that
$$\hat{\cJ}_{e}^{v_{\epsilon}}(x)\leq \hat{\cJ}_{e}^{*}(x) + \epsilon\,.$$ This completes the proof\,. 
\end{proof}
%%%%%%%%%%%%%%%%%%%%%%%%%%%%%%%%%%%%%%%%%%%%%%%%%%%%%%%%
\section*{Conclusion}
We studied the approximation of optimal control strategies by smooth/Lipschitz continuous policies for both finite/infinite horizon optimal control problems in a wide class of controlled diffusions in the entire space $\Rd$\,. Under a mild set of assumptions, we demonstrate that smooth/Lipschitz continuous policies are dense in the Markov/stationary Markov policies space under the Borkar topology. Next, by taking advantage of the continuity of costs with finite horizon/infinite horizon as a function of policies in the Markov/stationary Markov policy space (respectively) under the Borkar topology, we demonstrate that optimal policies can be closely approximated by smooth/Lipschitz continuous policies. Similar regularity results can be carried out in discrete-time under the ergodic cost criteria \cite{yuksel2023borkar}.

\subsection*{Acknowledgement}
Research was partially supported by the Natural Sciences and Engineering Research Council of Canada (NSERC)\,.
%%%%%%%%%%%%%%%%%%%%%%%%%%%%%%%%%%%%%%%%%%%%%%%%%%%%%%%%%%%%%%%%%%%%%%%%%%%%%%%%
%\bibliographystyle{plain}
\bibliography{Quantization,SerdarBibliography}

\begin{bibdiv}
\begin{biblist}

\bib{Bor89}{article}{
      author={Borkar, V.~S.},
       title={A topology for markov controls},
        date={1989},
     journal={Applied Mathematics and Optimization},
      volume={20},
       pages={55\ndash 62},
}

\bib{ABG-book}{book}{
      author={Arapostathis, A.},
      author={Borkar, V.~S.},
      author={Ghosh, M.~K.},
       title={Ergodic control of diffusion processes},
      series={Encyclopedia of Mathematics and its Applications},
   publisher={Cambridge University Press},
     address={Cambridge},
        date={2012},
      volume={143},
      review={\MR{2884272}},
}

\bib{pradhan2022near}{article}{
      author={Pradhan, S.},
      author={Y{\"u}ksel, S.},
       title={{Continuity of cost in Borkar control topology and implications
  on discrete space and time approximations for controlled diffusions under
  several criteria}},
        date={2024},
     journal={Electronic Journal of Probability},
      volume={29},
       pages={1 \ndash  32},
}

\bib{fleming2012deterministic}{book}{
      author={Fleming, W.~H.},
      author={Rishel, R.~W.},
       title={Deterministic and stochastic optimal control},
}

\bib{CainesMeanField1}{article}{
      author={Huang, M.},
      author={Malham\'e, R.P.},
      author={Caines, P.E.},
       title={Large population stochastic dynamic games: Closed-loop
  {M}c{K}ean-{V}lasov systems and the {N}ash certainty equivalence principle},
        date={2006},
     journal={Special issue in honour of the 65th birthday of Tyrone Duncan,
  Communications in Information and Systems},
      volume={6},
       pages={221\ndash 252},
}

\bib{JZ-JMLR23}{article}{
      author={Jia, Yanwei},
      author={Zhou, Xun~Yu},
       title={q-learning in continuous time},
        date={2023},
     journal={Journal of Machine Learning Research},
      volume={24},
      number={161},
       pages={1\ndash 61},
         url={http://jmlr.org/papers/v24/22-0755.html},
}

\bib{lechner2022stability}{inproceedings}{
      author={Lechner, M.},
      author={{\v{Z}}ikeli{\'c}, {\DJ}.},
      author={Chatterjee, K.},
      author={Henzinger, T.~A.},
       title={Stability verification in stochastic control systems via neural
  network supermartingales},
        date={2022},
   booktitle={Proceedings of the aaai conference on artificial intelligence},
      volume={36},
       pages={7326\ndash 7336},
}

\bib{richards2018lyapunov}{inproceedings}{
      author={Richards, S.~M.},
      author={Berkenkamp, F.},
      author={Krause, A.},
       title={The lyapunov neural network: Adaptive stability certification for
  safe learning of dynamical systems},
organization={PMLR},
        date={2018},
   booktitle={Conference on robot learning},
       pages={466\ndash 476},
}

\bib{chang2019neural}{article}{
      author={Chang, Y.-C.},
      author={Roohi, N.},
      author={Gao, S.},
       title={Neural lyapunov control},
        date={2019},
     journal={Advances in neural information processing systems},
      volume={32},
}

\bib{reisinger2023linear}{article}{
      author={Reisinger, C.},
      author={Stockinger, W.},
      author={Zhang, Y.},
       title={Linear convergence of a policy gradient method for some finite
  horizon continuous time control problems},
        date={2023},
     journal={SIAM Journal on Control and Optimization},
      volume={61},
      number={6},
       pages={3526\ndash 3558},
}

\bib{pradhanselkyuksel2023}{article}{
      author={Pradhan, S.},
      author={Selk, Z.},
      author={Y\"uksel, S.},
       title={Robustness of optimal controlled diffusions with near-brownian
  noise via rough paths theory},
        date={2023},
     journal={arXiv:2310.09967},
}

\bib{fazlyab2019efficient}{article}{
      author={Fazlyab, M.},
      author={Robey, A.},
      author={Hassani, H.},
      author={Morari, M.},
      author={Pappas, G.},
       title={Efficient and accurate estimation of lipschitz constants for deep
  neural networks},
        date={2019},
     journal={Advances in neural information processing systems},
      volume={32},
}

\bib{KD92}{book}{
      author={Kushner, H.~J.},
      author={Dupuis, P.~G.},
       title={Numerical methods for stochastic control problems in continuous
  time},
   publisher={Springer-Verlag, Berlin, New York,},
        date={2001},
}

\bib{KH77}{book}{
      author={Kushner, H.~J.},
       title={Probability methods for approximations in stochastic control and
  for elliptic equations},
      series={Math. Sci. Eng.},
   publisher={Academic Press, New York},
        date={1977},
      volume={129},
}

\bib{KH01}{book}{
      author={Kushner, H.~J.},
       title={Heavy traffic analysis of controlled queueing and communication
  networks},
      series={Stoch. Model. Appl. Probab.},
   publisher={Springer-Verlag, New York},
        date={2001},
      volume={47},
}

\bib{Adams}{book}{
      author={Adams, R.~A.},
       title={Sobolev spaces},
   publisher={Academic Press, New York},
        date={1975},
}

\bib{BL84-book}{book}{
      author={Bensoussan, A.},
      author={Lions, J.~L.},
       title={Impulse control and quasi-variational inequalities},
   publisher={Bristol: Gauthier-Villars},
        date={1984},
}

\bib{BG90}{article}{
      author={Borkar, V.~S.},
      author={Ghosh, M.~K.},
       title={Controlled diffusions with constraints},
        date={1990},
     journal={Journal of Mathematical Analysis and Applications},
      volume={152},
      number={1},
       pages={88\ndash 108},
         url={https://doi.org/10.1016/0022-247X(90)90094-V},
}

\bib{HP09-book}{book}{
      author={Pham, H.},
       title={Continuous-time stochastic control and applications with
  financial applications},
      series={Stochastic Modelling and Applied Probability},
   publisher={Springer},
        date={2009},
      volume={61},
}

\bib{B05Survey}{article}{
      author={Borkar, V.~S.},
       title={Controlled diffusion processes},
        date={2005},
     journal={Probab. Surveys},
      volume={2},
       pages={213\ndash 244},
}

\bib{yuksel2023borkar}{article}{
      author={Y{\"u}ksel, S.},
       title={On {B}orkar and {Y}oung relaxed control topologies and continuous
  dependence of invariant measures on control policy},
        date={2024},
     journal={SIAM Journal on Control and Optimization (to appear);
  arXiv:2304.07685},
}

\end{biblist}
\end{bibdiv}

\end{document}